\documentclass[a4paper]{amsart}

\usepackage[l2tabu,orthodox,abort]{nag}

\usepackage[british]{babel}

\usepackage{mlmodern}
\DeclareFontFamily{OMX}{mlmex}{}
\DeclareFontShape{OMX}{mlmex}{m}{n}{<->mlmex10}{}
\usepackage[T1]{fontenc}
\usepackage[babel,stretch=10,shrink=10,verbose=errors,selected]{microtype}
\usepackage[strict,english=british]{csquotes}

\usepackage{float}
\usepackage{booktabs}
\usepackage{array}

\usepackage{amsmath,amssymb,amsthm}
\usepackage{mathtools}
\usepackage[all]{onlyamsmath}
\usepackage{mleftright}

\usepackage{tikz}
\usetikzlibrary{cd}

\usepackage[
	backend=biber,
	style=numeric-comp,
	sorting=nyt,
	sortcites=true,
	maxnames=5,
	giveninits=true,
	date=year,
	useprefix=true
]{biblatex}
\DeclareFieldFormat{eid}{no.~#1}
\usepackage{enumitem}

\usepackage[nobiblatex]{xurl}
\usepackage[breaklinks=true,pdflang=en-GB,colorlinks=true]{hyperref}
\colorlet{citecolor}{green!75!black}
\colorlet{linkcolor}{red!75!black}
\colorlet{urlcolor}{blue!75!black}
\hypersetup{citecolor=citecolor,linkcolor=linkcolor,urlcolor=urlcolor}
\usepackage{zref-clever}
\zcsetup{cap,nameinlink=false}
\usepackage{keytheorems}

\usepackage{orcidlink}
\usepackage{ellipsis}

\newkeytheorem{theorem}[parent=section,style=plain]
\newkeytheorem{lemma}[sibling=theorem,parent=section,style=plain]
\newkeytheorem{corollary}[sibling=theorem,parent=section,style=plain]
\newkeytheorem{proposition}[sibling=theorem,parent=section,style=plain]
\newkeytheorem{conjecture}[sibling=theorem,parent=section,style=plain]

\newkeytheorem{definition}[sibling=theorem,parent=section,style=definition]
\newkeytheorem{example}[sibling=theorem,parent=section,style=definition]
\newkeytheorem{question}[sibling=theorem,parent=section,style=definition]

\DeclareMathOperator{\Hom}{Hom}
\DeclareMathOperator{\Aut}{Aut}
\DeclareMathOperator{\Inn}{Inn}

\DeclareMathOperator{\id}{id}

\newcommand{\normalsub}{\trianglelefteq}
\newcommand{\fisub}{\leq_f}
\newcommand{\finormalsub}{\normalsub_f}

\DeclareMathOperator{\BS}{BS} 
\DeclareMathOperator{\GL}{GL} 
\DeclareMathOperator{\SL}{SL} 

\DeclareMathOperator{\Coin}{Coin}

\newcommand{\NN}{\mathbb{N}}
\newcommand{\ZZ}{\mathbb{Z}}
\newcommand{\QQ}{\mathbb{Q}}

\newcommand{\nest}[1]{\mathcal{#1}}
\newcommand{\prenest}[1]{\nest{#1}_\times}

\newcommand{\bvarphi}{{\bar{\varphi}}}
\newcommand{\bpsi}{{\bar{\psi}}}
\newcommand{\tvarphi}{{\tilde{\varphi}}}
\newcommand{\tpsi}{{\tilde{\psi}}}

\newcommand{\topPF}{\tau_{\scriptscriptstyle \mathrm{PF}}}
\newcommand{\topFI}{\tau_{\scriptscriptstyle \mathrm{FI}}}

\title{Twisted conjugacy and separability}
\author[Sam Tertooy]{Sam Tertooy\,\orcidlink{0000-0002-5750-9153}}
\date{\today}
\address{KU Leuven, Kulak Kortrijk Campus\\
	E.~Sabbelaan 53\\
	8500 Kortrijk\\
	Belgium}
\email{\href{mailto:sam.tertooy@kuleuven.be}{sam.tertooy@kuleuven.be}}
\urladdr{\url{https://stertooy.github.io}}


\begin{document}
	
\begin{abstract}
	A group \(G\) is twisted conjugacy separable if for every automorphism \(\varphi\), distinct \(\varphi\)-twisted conjugacy classes can be separated in a finite quotient. Likewise, \(G\) is completely twisted conjugacy separable if
	for any group \(H\) and any two homomorphisms \(\varphi,\psi\) from \(H\) to \(G\), distinct \((\varphi,\psi)\)-twisted conjugacy classes can be separated in a finite quotient. We study how these properties behave with respect to taking subgroups, quotients and finite extensions, and compare them to other notions of separability in groups. Finally, we show that for polycyclic-by-nilpotent-by-finite groups, being completely twisted conjugacy separable is equivalent to all quotients being residually finite.
\end{abstract}
	
\maketitle

	\begin{center}
	This is an Accepted Manuscript of an article published by De Gruyter in Journal of Group Theory on 5 Dec 2024, available online: \href{https://doi.org/nv2p}{https://doi.org/nv2p}.
\end{center}

\section{Introduction}
Let \(G\) and \(H\) be groups and let \(\varphi,\psi\colon H \to G\) be two homomorphisms. Two elements \(g_1,g_2\) of \(G\) are said to be \emph{\((\varphi,\psi)\)-twisted conjugate} if there exists some \(h \in H\) such that \(g_1 = \psi(h) g_2 \varphi(h)^{-1}\). This creates an equivalence relation, the classes of which are called \emph{\((\varphi,\psi)\)-twisted conjugacy classes}; the class of \(g \in G\) is denoted by \([g]_{\varphi,\psi}\). If \(G = H\) and \(\psi = \id_G\), we omit \(\psi\) and just write `\(\varphi\)-twisted conjugacy' and \([g]_\varphi\). If moreover \(\varphi = \id_G\) as well, we have the usual conjugacy relation and we write \([g]\) for the conjugacy class of \(g \in G\).

In \cite{steb76-a}, Stebe introduced the notion of nests in a group \(G\). A subset of \(G \times G\) is called a \emph{pre-nest} if it contains \((ac^{-1},d^{-1}b)\) whenever it contains both \((a,b)\) and \((c,d)\). The set of all elements \(ab \in G\) corresponding to the elements \((a,b)\) of a pre-nest is called a \emph{nest} in \(G\). Examples of nests include subgroups, products of two subgroups and twisted conjugacy classes of the identity.

A subset \(S\) of a group \(G\) is called \emph{separable} if for every \(g \notin S\), there is a finite index normal subgroup \(N\) of \(G\) such that \(g \notin SN\). This gives rise to many `separability properties', e.g.\@ a group \(G\) is called \emph{residually finite} if all singletons \(\{g\}\) are separable, \emph{strongly residually finite} if all normal subgroups are separable, \emph{extended residually finite} if all subgroups are separable, \emph{conjugacy separable} if all conjugacy classes are separable and \emph{residually finite with respect to nests} if all nests are separable.

In \cite{dp18-a}, Deré and Pengitore defined a group \(G\) to be \emph{twisted conjugacy separable} if for all \(\varphi \in \Aut(G)\) and all \(g \in G\), the twisted conjugacy class \([g]_\varphi\) is separable. It should be noted that the term `twisted conjugacy separable' first occurred in the works of Fel'shtyn and Troitsky \cite{ft06-a,ft07-a}. However, their definition differs from the one given by Deré and Pengitore, which is the one we shall use.

In this paper, we introduce and study \emph{complete twisted conjugacy separability}. A group \(G\) is said to have this property if for all groups \(H\), all \(\varphi,\psi \in \Hom(H,G)\) and all \(g \in G\), the twisted conjugacy class \([g]_{\varphi,\psi}\) is separable. Unlike (twisted) conjugacy separability, this property is retained when taking subgroups, quotients and finite extensions. Our main results are the following.

\begin{theorem}[manual-num=A]
	\label{thm:mainthmAintro}
	A group is residually finite with respect to nests if and only if it is completely twisted conjugacy separable.
\end{theorem}

\begin{theorem}[manual-num=B]
	\label{thm:mainthmBintro}
	For a polycyclic-by-nilpotent-by-finite group \(G\), the following are equivalent:
	\begin{enumerate}[label=(\alph*)]
		\item \(G\) is strongly residually finite;
		\item \(G\) is extended residually finite;
		\item \(G\) is residually finite with respect to nests;
		\item \(G\) is completely twisted conjugacy separable.
	\end{enumerate}
\end{theorem}

\zcref{thm:mainthmAintro} shows that residual finiteness with respect to nests, or even nests in general, can be studied using twisted conjugacy as a framework. \zcref{thm:mainthmBintro} extends a result by Menth \cite{ment02-a}, who proved that strong and extended residual finiteness are equivalent for nilpotent-by-finite groups.

We finish this section by agreeing on some notation. For a group \(G\), we will always use \(\bar{G}\) and \(\tilde{G}\) to denote quotients of \(G\). For an element \(g\), we use \(\bar{g}\) and \(\tilde{g}\) to denote the projections of \(g\) to \(\bar{G}\) and \(\tilde{G}\) respectively. For a homomorphism \(\varphi \in \Hom(H,G)\), whenever we write \(\bvarphi \in \Hom(H,\bar{G})\), we mean the homomorphism induced by \(\varphi\), i.e.\@ \(\bvarphi \coloneq p \circ \varphi\) with \(p\colon G \to \bar{G}\) the quotient map. We will also encounter \(\hat{G}\), however, the `hat' will always indicate the profinite completion of \(G\) and never a quotient of \(G\). We will use \(H \fisub G\) and \(N \finormalsub G\) to indicate a finite index subgroup \(H\) and a finite index normal subgroup \(N\) respectively. The inner automorphism of \(G\) that conjugates every element by \(g \in G\), will be denoted by \(\iota_g\).

\section{Separability properties}
\label{sec:groupprops}
A subset \(S\) of a group \(G\) is called \emph{separable} if any of the following equivalent conditions hold:
\begin{itemize}
	\item For each \(g \in G \setminus S\), there exists some \(N \finormalsub G\) such that \(g \notin SN\);
	\item For each \(g \in G \setminus S\), there exists a homomorphism \(p\) from \(G\) onto a finite group such that \(p(g) \notin p(S)\);
	\item \(S\) is closed in the finite-index topology of \(G\).
\end{itemize}
The finite-index topology mentioned here is the topology on a group obtained by taking as a basis the set of all cosets (both left and right) of all finite index subgroups. Any group becomes a topological group when equipped with this topology. This topology is often called the ``profinite topology'', however, as we will be dealing with profinite groups and their topology later in this paper, we avoid this name to prevent any potential confusion between the two.

We now give an overview of various separability properties. We give their definition, mention whether they are retained or not under taking subgroups, quotients and finite extensions, and discuss how they are related to one another. Only results available in the literature are considered in this \zcref[nocap,noref]{sec:groupprops}.

\begin{definition}
	\label{def:resfin}
	A group \(G\) is called \emph{residually finite} (RF for short) if every singleton \(\{g\} \subseteq G\) is separable.
\end{definition}

There are many equivalent ways to define residual finiteness. It suffices to demand that \(\{1_G\}\) is separable, or, one can demand that every finite subset of \(G\) is separable. From a topological point of view, \(G\) is residually finite if and only if its finite-index topology is Hausdorff.

The family of RF-groups is subgroup-closed and closed under taking finite extensions, but not quotient-closed. For example, the group of dyadic rationals \(\ZZ[\frac{1}{2}]\) is residually finite, yet its quotient \(\ZZ[\frac{1}{2}]/\ZZ\) is the quasicyclic group \(\ZZ({2^\infty})\), which is not residually finite.

\begin{definition}
	A group \(G\) is called \emph{conjugacy separable} (CS) if every conjugacy class \([g]\) is separable.
\end{definition}
In a conjugacy separable group \(G\), if two elements \(g,h \in G\) are non-conjugate, then their images \(\bar{g},\bar{h}\) are non-conjugate in some finite quotient \(\bar{G}\). Many examples of conjugacy separable groups are known, including all polycyclic-by-finite groups \cite{reme69-a,form76-a}. Unfortunately, this property behaves quite badly: it is neither subgroup-closed \cite[Thm.~1.1]{mm12-a}, nor quotient-closed (again, \(\ZZ[\frac{1}{2}]\) is a counterexample), nor closed under taking finite extensions \cite{gory86-a}.

\begin{definition}
	\label{def:tcs}
	A group \(G\) is called \emph{twisted conjugacy separable} (TCS) if for every automorphism \(\varphi \in \Aut(G)\) and every element \(g \in G\), the twisted conjugacy class \([g]_\varphi\) is separable.
\end{definition}

This definition was given by Deré and Pengitore in \cite{dp18-a}. Little is known about how this property behaves with respect to subgroups, quotients and finite extensions, hence we will study this property extensively in \zcref{sec:tcs}.

The original definition of twisted conjugacy separability, which differs from \zcref{def:tcs}, was given by Fel'shtyn and Troitsky in \cite{ft07-a} and further developed in \cite{ft06-a}. Their definition says that a group \(G\) is twisted conjugacy separable if for every \(\varphi \in \Aut(G)\) with finitely many twisted conjugacy classes and every \(g,h \in G\) with \([g]_\varphi \neq [h]_\varphi\), there exists a \(\varphi\)-invariant \(N \finormalsub G\) such that \([\bar{g}]_\bvarphi \neq [\bar{h}]_\bvarphi\), where \(\bvarphi\) is the induced automorphism on \(\bar{G} \coloneq G/N\). In the sequel, we will exclusively use \zcref{def:tcs}.

\begin{definition}
	A group \(G\) is called \emph{strongly residually finite} (SRF) if every normal subgroup of \(G\) is separable.
\end{definition}

Equivalently, \(G\) is an SRF-group if and only if every quotient of \(G\) is residually finite. These groups also go by the names ``QRF-groups'' and ``(RF)\textsuperscript{Q}-groups''. Strong residual finiteness was studied in FC-groups \cite{ko03-a} and nilpotent groups \cite{ment02-a}; any finitely generated abelian-by-polycyclic-by-finite group (and in particular any finitely generated metabelian group) is SRF \cite[Thm.~3]{jate74-a}. The family of SRF-groups is closed under taking finite extensions, but in contrast with RF-groups, it is quotient-closed yet not subgroup-closed. For example, the Baumslag-Solitar group \(\BS(1,2)\), being finitely generated metabelian, is SRF, but its derived subgroup \(\ZZ[\frac{1}{2}]\) is not.

\begin{definition}
	\label{def:extresfin}
	A group \(G\) is called \emph{extended residually finite} (ERF) if every subgroup of \(G\) is separable.
\end{definition}

We refer to \cite{rrv09-a} for a summary on ERF-groups. The family of ERF-groups is closed under taking finite extensions and is both subgroup- and quotient-closed. A well-known result of Mal'cev says that polycyclic-by-finite groups are ERF \cite{malc58-a}; they remain the only known finitely generated ERF-groups. Characterisations for nilpotent \cite{ment02-a}, FC \cite{rrv09-a}, FC\textsuperscript{*} \cite{rrv11-a} and linear \cite{wehr11-a} ERF-groups are known.

In \cite{steb76-a}, Stebe introduced \emph{nests} and \emph{residual finiteness with respect to nests}, specifically with the aim of uniting conjugacy separability and extended residual finiteness in a single separability property.

\begin{definition}
	Let \(G\) be a group. A \emph{pre-nest} \(\prenest{N}\) is a non-empty subset of \(G \times G\) such that if both \((a,b) \in \prenest{N}\) and \((c,d) \in \prenest{N}\), then also \((ac^{-1},d^{-1}b) \in \prenest{N}\). A \emph{nest} \(\nest{N}\) in \(G\) is the image of a pre-nest \(\prenest{N}\) in \(G \times G\) under the multiplication map \(G \times G \to G\colon (g_1,g_2) \mapsto g_1g_2\).
\end{definition}

It can easily be seen that any pre-nest must contain \((1,1)\) and hence any nest contains the identity. We also remark that Stebe did not use the term ``pre-nest'' and instead called these sets ``nests'' as well. Below are some examples of nests, all of which appeared in  \cite{steb76-a,steb87-a}.

\begin{example}
	\label{ex:nests}
	The following subsets of a group \(G\) are nests:
	\begin{enumerate}[label=(\arabic*)]
		\item any subgroup \(H \leq G\);
		\item any double coset \(HK\) for subgroups \(H,K \leq G\);
		\item the set of commutators \(\{ [h,g] \mid h \in G\} \) for any fixed \(g \in G\);
		\item the set \(\{h\varphi(h)^{-1} \mid h \in G\}\) for any fixed \(\varphi \in \Aut(G)\).
	\end{enumerate}
\end{example}

\begin{definition}
	A group \(G\) is called \emph{residually finite with respect to nests} (NRF) if every nest in \(G\) is separable.
\end{definition}

The family of NRF-groups is closed under taking subgroups \cite[Thm.~4]{steb76-a} and under taking finite extensions \cite[Thm.~3]{steb87-a}; the behaviour with respect to quotients will be discussed in \zcref{sec:NestsAndCTCS}.

\zcref{ex:nests} (1) shows that being NRF implies being extended residually finite. From (2), we can deduce that in NRF-groups any double coset \(HgK\) is separable, since \(x \in HgK\) if and only if \(xg^{-1} \in HgKg^{-1}\).
Similarly, (3) implies that an NRF-group is conjugacy separable, as \(x = hgh^{-1}\) if and only if \(xg^{-1} = [h,g]\). Finally, (4) implies that an NRF-group is twisted conjugacy separable: \(x = hg\varphi(h)^{-1}\) if and only if \(xg^{-1} = h(\iota_g\varphi)(h)^{-1}\).

We summarise how these notions of separability behave with respect to taking 
subgroups, quotients and finite extensions in \zcref{tbl:rfproperties}. Their relation to one another is illustrated by the diagram below.
\begin{equation}
	\label{eq:diagram}
    \begin{tikzcd}[row sep=small,column sep=small,ampersand replacement=\&]
	   \& \text{TCS}\arrow[r,phantom,"\subsetneq"] \& \text{CS}\arrow[dr,phantom,"\subsetneq", sloped, start anchor=south east] \& \\
	   \text{NRF}\arrow[dr,phantom,"\subseteq" description, sloped]\arrow[ur,phantom,"\subseteq" description, sloped]\& \& \&\text{RF} \\
	   \& \text{ERF}\arrow[r,phantom,"\subsetneq"]\&\text{SRF} \arrow[ur,phantom,"\subsetneq" description, sloped]\& 
    \end{tikzcd}
\end{equation}
\begin{table}[hb]
	\caption{Separability properties of groups}
	\label{tbl:rfproperties}
	\begin{tabular}{cccc}
		\toprule
		separability&  subgroup- & quotient- & closed under \\
		property & closed&closed&finite extensions\\
		\midrule
		RF &Yes	& No & Yes  \\
		CS &No & No & No\\
		TCS  &? & ? & ?\\
		SRF &No &	Yes & Yes\\
		ERF  &Yes & Yes & Yes\\
		NRF &Yes & ? & Yes\\
		\bottomrule
	\end{tabular}
\end{table}

\begin{proposition}
	In Diagram \eqref{eq:diagram}, the two inclusions in the middle and the two inclusions to the right are indeed strict.
\end{proposition}
\begin{proof}
We give examples of groups that illustrate that these inclusions are strict.
\begin{itemize}
	\item CS \(\subsetneq\) RF: the general and special linear groups \(\GL_n(\ZZ)\), \(\SL_n(\ZZ)\) with \(n \geq 3\) are RF but not CS \cite[Thm.~3]{steb72-a}.
	\item TCS \(\subsetneq\) CS: if \(G\) is a centreless finite group, the direct sum of countably infinitely many copies of \(G\) is CS but not TCS \cite[Sec.~3]{troi20-a}.
	\item SRF \(\subsetneq\) RF: the group of dyadic rationals \(\ZZ[\frac{1}{2}]\) is RF, but its quasicyclic quotient \(\ZZ[\frac{1}{2}]/\ZZ\) is not. Thus \(\ZZ[\frac{1}{2}]\) is not SRF.
	\item ERF \(\subsetneq\) SRF: we stated earlier that \(\BS(1,2)\) is SRF, but its derived subgroup \(\ZZ[\frac{1}{2}]\) is not. If \(\BS(1,2)\) were ERF, then its derived subgroup would also be ERF and therefore SRF.\qedhere
\end{itemize}
\end{proof}

\section{Profinite groups}

The class of profinite groups will be a rich source of (counter)examples when studying twisted conjugacy separability. Below, we give a short overview of the definitions and properties we will need later, based on \cite[Sec.~2.1 \& Sec.~4.2]{rz10-a}.

\begin{definition}
	A \emph{profinite group} is a topological group that is compact, Hausdorff and totally disconnected.
\end{definition}

\begin{proposition}
	\label{prop:profiniteproperties}
	For any profinite group, one has:
	\begin{itemize}
		\item The identity has a local basis consisting of open normal subgroups;
		\item Subgroups are open if and only if they are closed and have finite index.
	\end{itemize}
\end{proposition}

Let \(G\) be a group and consider the collection of its finite index normal subgroups. If \(M,N \finormalsub G\) and \(M \leq N\), let \(p_{MN}\colon G/M \to G/N\) be the natural projection. Then \(\{G/N,p_{MN}\}\) is an inverse system of (finite) groups. 

\begin{definition}
	The \emph{profinite completion} \(\hat{G}\) of a group \(G\) is the projective limit
	\begin{equation*}
		\hat{G} \coloneq \varprojlim_{N \finormalsub G} G/N,
	\end{equation*}
	equipped with the limit topology.
\end{definition}

Thus, the profinite completion \(\hat{G}\) of \(G\) is a subgroup of the direct product \(\prod_{N \finormalsub G} G/N\). The map
\begin{equation*}
	\iota\colon G \to \hat{G} \subseteq \prod_{N \finormalsub G} G/N \colon g \mapsto (gN)_{N}
\end{equation*}
is called the \emph{canonical homomorphism} from \(G\) to \(\hat{G}\). 

\begin{proposition}
	Let \(G\) be a group, \(\hat{G}\) its profinite completion and \(\iota\colon G \to \hat{G}\) the canonical homomorphism. Then:
	\begin{itemize}
		\item \(\iota(G)\) is dense in \(\hat{G}\);
		\item \(\iota\) is injective if and only if \(G\) is residually finite.
	\end{itemize}
\end{proposition}

Let \(G\) be a group and \(\topPF\) a topology such that \(G\) equipped with \(\topPF\) is a profinite group. Let \(U\) be an open set in \(G\) with respect to \(\topPF\). For any \(g \in U\), the set \(g^{-1}U\) is open and contains the identity \(1_G\). Thus, there exists an open normal subgroup \(N_g\) contained in \(g^{-1}U\). But then
\begin{equation*}
	U = \bigcup_{g \in U} gN_g,
\end{equation*}
and each \(N_g\) has finite index \(G\). Then \(U\) is also open in the finite-index topology  \(\topFI\) on \(G\). Thus, in general, the finite-index topology \(\topFI\) on \(G\) is finer than \(\topPF\). Those profinite groups for which both topologies coincide, which happens exactly when all finite index subgroups are open, are of particular interest to us.

\begin{definition}
	A profinite group \(G\) is called \emph{strongly complete} if it satisfies any of the following equivalent conditions:
	\begin{enumerate}[label=(\arabic*)]
		\item Every finite index subgroup of \(G\) is open;
		\item \(G\) coincides with its profinite completion \(\hat{G}\);
		\item Every homomorphism from \(G\) to a profinite group is continuous.
	\end{enumerate}
\end{definition}

We say that a topological group \(G\) is \emph{topologically finitely generated} if it admits a finitely generated subgroup \(\langle g_1,\ldots,g_n\rangle\) that is dense. In \cite{ns07-a}, Nikolov and Segal proved the following celebrated \zcref[nocap,noref]{thm:fgprofiniteissc}.

\begin{theorem}
	\label{thm:fgprofiniteissc}
	Topologically finitely generated profinite groups are strongly complete.
\end{theorem}

If \(G\) is a finitely generated group and \(\hat{G}\) is its profinite completion, then \(\iota(G)\) is finitely generated and dense in \(\hat{G}\). Thus, \(\hat{G}\) is topologically finitely generated and therefore strongly complete.

\section{Twisted conjugacy separability}
\label{sec:tcs}

The definition of twisted conjugacy separability was already given in \zcref{sec:groupprops}. We shall start this section by giving some related definitions, and then study how twisted conjugacy separability behaves with respect to taking subgroups, quotients and finite extensions.

\begin{definition}
	\label{def:twiconsep}
	A group \(G\) is \emph{\((\varphi,\psi)\)-twisted conjugacy separable} if every \((\varphi,\psi)\)-twisted conjugacy class is separable, where \(\varphi,\psi \in \Hom(H,G)\) for some group \(H\).
\end{definition}

Equivalently, for any pair \(g,k\) of non-\((\varphi,\psi)\)-twisted conjugate elements of \(G\), there exists a finite quotient \(\bar{G}\) such that \(\bar{g},\bar{k}\) are non-\((\bvarphi,\bpsi)\)-twisted conjugated, where \(\bvarphi,\bpsi \in \Hom(H,\bar{G})\). 

\begin{definition}
	Let \(G\) and \(H\) be groups. We say that \(G\) is \emph{\(H\)-twisted conjugacy separable} if \(G\) is \((\varphi,\psi)\)-twisted conjugacy separable for all \(\varphi,\psi \in \Hom(H,G)\).
\end{definition}

We are now ready to study twisted conjugacy separability on profinite groups. The following \zcref[nocap,noref]{lem:tccclosed} can be found in e.g.\@ \cite[Lem.~2.3(1)]{fk22-a}.

\begin{lemma}
	\label{lem:tccclosed}
	Let \(G\) and \(H\) be profinite groups and let \(\varphi,\psi \in \Hom(H,G)\) be continuous homomorphisms. Then the \((\varphi,\psi)\)-twisted conjugacy classes are closed.
\end{lemma}
\begin{proof}
	Let \(g \in G\). The map \(H \to G\colon h \mapsto \psi(h)g\varphi(h)^{-1}\) is continuous, and its image is exactly \([g]_{\varphi,\psi}\). Because \(H\) is compact, so is its continuous image \([g]_{\varphi,\psi}\). Since \(G\) is Hausdorff, \([g]_{\varphi,\psi}\) is then closed.
\end{proof}

\begin{theorem}
	\label{thm:profinite}
	Let \(G\) and \(H\) be profinite groups, with \(H\) being strongly complete. Then \(G\) is \(H\)-twisted conjugacy separable.
\end{theorem}
\begin{proof}
	Let \(\topPF\) be a topology on \(G\) such that \(G\) equipped with \(\topPF\) is profinite. Let \(\varphi,\psi \in \Hom(H,G)\), which are necessarily continuous since \(H\) is strongly complete. By \zcref{lem:tccclosed}, all \((\varphi,\psi)\)-twisted conjugacy classes are closed with respect to \(\topPF\). But \(\topPF\) is coarser than the finite-index topology \(\topFI\) on \(G\). Thus, the twisted conjugacy classes are also closed with respect to \(\topFI\) and therefore separable.
\end{proof}

 When considering automorphisms of a strongly complete profinite group \(G\), we thus get the next result.

\begin{corollary}
	Strongly complete profinite group are twisted conjugacy separable.
\end{corollary}

We now fill in some of the gaps left in \zcref{sec:groupprops}. First, we aim to show that a subgroup of a TCS-group need not be TCS.

\begin{example}
	\label{ex:tcssub}
	Let \(G = \SL_n(\ZZ)\) for some \(n \geq 3\), which is a finitely generated RF-group that is not conjugacy separable. In fact, by the Margulis Normal Subgroup Theorem \cite[Thm.~4]{marg91-a}, every normal subgroup of \(G\) is either finite or has finite index in \(G\), and consequently \(G\) is SRF. Let \(\hat{G}\) be its profinite completion and \(\iota\colon G \to \hat{G}\) the canonical homomorphism. Then \(\hat{G}\) is strongly complete and therefore TCS, but its subgroup \(\iota(G) \cong G\) is not even conjugacy separable.
\end{example}

Similarly, we illustrate how the quotient of a TCS-group need not be TCS.

\begin{example}
	\label{ex:tcsquo}
	Let \(p\) be a prime and consider the group of \(p\)-adic integers \(\ZZ_p\), which is a topologically finitely generated profinite group. Thus, it is strongly complete and hence TCS. Let \(\QQ_p\) denote the \(p\)-adic rationals, i.e.\@
	\begin{equation*}
		\QQ_p = \bigcup_{n \in \NN} \frac{1}{p^n} \ZZ_p,
	\end{equation*}
	which is a vector space over \(\QQ\). Using the axiom of choice, pick a \(\QQ\)-vector space epimorphism \(\QQ_p \to \QQ\). Its restriction to \(\ZZ_p\) is still surjective (see \cite[Sec.~1.2.1]{ns12-a}), so \(\ZZ_p\) has a quotient isomorphic to \(\QQ\). But \(\QQ\) is not even residually finite.
\end{example}

This example also shows that TCS-groups need not be SRF (and therefore certainly need not be NRF). Thus, the inclusion \(\mathrm{NRF} \subsetneq \mathrm{TCS}\) is strict.

The final question we would like to answer, is whether or not a finite extension of a TCS-group is again TCS. We provide a partial answer, which is strongly based on \cite[Thm.~5.2]{ft06-a}, but the general case will remain open.

\begin{proposition}
	Let \(N\) be a finite index normal subgroup of a group \(G\).
	\begin{enumerate}[label=(\alph*)]
		\item If \(N\) is TCS, then \(G\) is CS;
		\item If \(N\) is characteristic and TCS, then \(G\) is TCS.
	\end{enumerate}
\end{proposition}
\begin{proof}
	Let \(g \in G\). Suppose that \(x_1, \ldots, x_r\) are coset representatives of \(G/N\) with \(r \coloneq [G:N]\). It is easy to see that if a set is open (closed) in the finite-index topology on \(H\), then it is open (closed) in the finite-index topology on \(G\).
	
	For each \(i = 1,\ldots,r\), set \(g_i \coloneq x_igx_i^{-1}\) and let \(\varphi_i \in \Aut(N)\) be the restriction of \(\iota_{g_i} \in \Inn(G)\) to \(N\). We can write
	\begin{align*}
		[g] &= \{ hgh^{-1} \mid h \in G\}\\
		&= \bigcup_{i=1}^r \{nx_igx_i^{-1}n^{-1} \mid n \in N\} \\
		&= \bigcup_{i=1}^r \{ng_in^{-1} \mid n \in N\} \\
		&= \bigcup_{i=1}^r \{n\varphi_i(n)^{-1}g_i \mid n \in N\} \\
		&= \bigcup_{i=1}^r [1]_{\varphi_i} \cdot g_i.
	\end{align*}
	For each \(i\), \([1]_{\varphi_i}\) is a twisted conjugacy class in \(N\). Thus, each \([1]_{\varphi_i}\) is closed in \(N\) and therefore closed in \(G\). So \([g]\) is a finite union of closed sets and therefore closed itself, which proves part (a).
	
	For part (b), one may take an analogous approach. Let \(g \in G\) and \(\psi \in \Aut(G)\), set \(g_i \coloneq x_i g \psi(x_i)^{-1}\) and let \(\psi_i \in \Aut(N)\) be the restriction of \(\iota_{g_i}\psi\) to \(N\). We can then write
	\begin{equation*}
		[g]_\psi = \bigcup_{i=1}^r [1]_{\psi_i} \cdot g_i,
	\end{equation*}
	so indeed \([g]_\psi\) is closed in the finite-index topology on \(G\).
\end{proof}

\begin{question}
	\label{qtn:isTCSclosedFE}
	Does there exist a group \(G\) with finite index normal subgroup \(N\) such that \(N\) is TCS but \(G\) is not?
\end{question}

\section{Nests and complete twisted conjugacy separability}
\label{sec:NestsAndCTCS}

The main result obtained by Stebe with regards to nests, is that polycyclic-by-finite groups are residually finite with respect to nests \cite[Thm.~5]{steb87-a}. This result was used by Gonçalves and Wong to study twisted conjugacy. Using the constructions and arguments they provide in \cite[Sec.~3]{gw05-a}, one may extract the following theorem and proof.

\begin{theorem}
	Let \(G\) and \(H\) be polycyclic-by-finite groups. Then \(G\) is \(H\)-twisted conjugacy separable.
\end{theorem}
\begin{proof}
	Let \(\varphi,\psi \in \Hom(H,G)\) and let \(g_1, g_2 \in G\) with \([g_1]_{\varphi,\psi} \neq [g_2]_{\varphi,\psi}\). Consider the subgroups
	\begin{equation*}
		A \coloneq \{(h,\psi(h)) \mid h \in H\}, \quad B \coloneq \{(h,g_2\varphi(h)g_2^{-1}) \mid h \in H\}
	\end{equation*}
	of \(H \times G\). The double coset \(AB\) is a nest in \(H \times G\). Suppose, by contradiction, that \((1,g_1g_2^{-1}) \in AB\). Then there exist \(h_1,h_2 \in H\) such that
	\begin{equation*}
		(1,g_1g_2^{-1}) = (h_1h_2,\psi(h_1)g_2\varphi(h_2)g_2^{-1}).
	\end{equation*}
	Looking at the first component, we must have \(h_2 = h_1^{-1}\). Equality for the second component then becomes
	\begin{equation*}
		g_1g_2^{-1} = \psi(h_1)g_2\varphi(h_1)^{-1}g_2^{-1}.
	\end{equation*}
	If such an \(h_1\) existed, then \(g_1\) and \(g_2\) would be \((\varphi,\psi)\)-twisted conjugate, which is untrue by assumption. So \((1,g_1g_2^{-1}) \notin AB\). Let \(N \finormalsub H \times G\) such that
	\begin{equation*}
		(1,g_1g_2^{-1}) \notin ABN.
	\end{equation*}
	Then, setting \(M \coloneq G \cap N \finormalsub G\), we find that \(g_1 \notin [g_2]_{\varphi,\psi}M\).
\end{proof}

However, there is an alternative way to apply the separability of nests to \((\varphi,\psi)\)-twisted conjugacy classes.

\begin{proposition}
	\label{prop:tcc1isnest}
	Let \(G\) and \(H\) be groups and let \(\varphi,\psi \in \Hom(H,G)\). Then the twisted conjugacy class \([1]_{\varphi,\psi}\) is a nest in \(G\).
\end{proposition}
\begin{proof}
	The set \(\{ (\psi(h),\varphi(h)^{-1}) \mid h \in H \} \) is a pre-nest in \(G \times G\). Its corresponding nest in \(G\) is exactly \([1]_{\varphi,\psi}\).
\end{proof}

At first glance, this result does not appear to be all that useful. After all, we want to study separability of \emph{all} twisted conjugacy classes, and not just the class of the identity. The following \zcref[nocap,noref]{lem:equivto1isenough} and \zcref[nocap,noref]{prop:Htwistconj} show that studying the separability of the twisted conjugacy class of the identity is, in fact, sufficient.

\begin{lemma}
	\label{lem:equivto1isenough}
	Let \(G\) and \(H\) be groups, let \(g,k \in G\) and let \(\varphi,\psi \in \Hom(H,G)\). Then
	\begin{equation*}
		[g]_{\varphi,\psi} = [k]_{\varphi,\psi} \iff [gk^{-1}]_{\iota_k\varphi,\psi} = [1]_{\iota_k\varphi,\psi}.
	\end{equation*}
\end{lemma}

\begin{proposition}
	\label{prop:Htwistconj}
	Let \(G\) and \(H\) be groups. The following are equivalent:
	\begin{enumerate}[label=(\arabic*)]
		\item \(G\) is \(H\)-twisted conjugacy separable.
		\item For all \(\varphi,\psi \in \Hom(H,G)\), the twisted conjugacy class \([1]_{\varphi,\psi}\) is separable.
	\end{enumerate}
\end{proposition}
\begin{proof}
	It is easy to see that (1) implies (2). For the reverse implication, let \(g,k \in G\) and \(\varphi,\psi \in \Hom(H,G)\) with \([g]_{\varphi,\psi} \neq [k]_{\varphi,\psi}\). Applying \zcref{lem:equivto1isenough}, we find that \([gk^{-1}]_{\iota_k\varphi,\psi} \neq [1]_{\iota_k\varphi,\psi}\). Then there exists some\(N \finormalsub G\) such that \[[\bar{g}\bar{k}^{-1}]_{\iota_{\bar{k}}\bvarphi,\bpsi} \neq [\bar{1}]_{\iota_{\bar{k}}\bvarphi,\bpsi}\] in \(\bar{G} \coloneq G/N\), where \(\bvarphi, \bpsi \in \Hom(H,\bar{G})\). Applying \zcref{lem:equivto1isenough} again, we get \([\bar{g}]_{\bvarphi,\bpsi} \neq [\bar{k}]_{\bvarphi,\bpsi}\).
\end{proof}

Since polycyclic-by-finite groups are residually finite with respect to nests, and the twisted conjugacy class \([1]_{\varphi,\psi}\) is always a nest irrespective of the group \(H\) and the homomorphisms \(\varphi,\psi\colon H \to G\), \zcref{prop:Htwistconj} implies that polycyclic-by-finite groups are \(H\)-twisted conjugacy separable for every group \(H\). We give this property a name:

\begin{definition}
	A group \(G\) is \emph{completely twisted conjugacy separable (CTCS)} if it is \(H\)-twisted conjugacy separable for every group \(H\).
\end{definition}

At this point, it is clear that being residually finite with respect to nests implies complete twisted conjugacy separability. To prove the converse, we first need to understand the structure of pre-nests, which was described by Stebe in \cite[Thm.~1]{steb76-a}.

\begin{theorem}
	\label{thm:neststructure}
	Let \(S_1\) and \(S_2\) be subgroups of a group \(G\). Let \(K_i\) be a normal subgroup of \(S_i\) for \(i=1,2\) such that \(S_1/K_1\) is isomorphic to \(S_2/K_2\), and let \(\theta\colon S_1/K_1 \to S_2/K_2\) be an isomorphism. The set of all elements of \(G \times G\) of the form \((ac,\theta(c^{-1})b)\), with \(a\) in \(K_1\), \(b\) in \(K_2\) and \(c\) a coset representative of \(S_1\) modulo \(K_1\) is a pre-nest in \(G \times G\). Conversely, every pre-nest in \(G \times G\) may be obtained by this construction.
\end{theorem}

It should be noted that there is some slight abuse of notation in the \zcref[nocap,noref]{thm:neststructure} above: ``\(\theta(c^{-1})\)'' should be interpreted as ``a coset representative of \(\theta(c^{-1}K_1)\)''. In the case of a (\(\varphi,\psi\))-twisted conjugacy class with \(\varphi,\psi \in \Hom(H,G)\), the groups in question are given by

\begin{align*}
	S_1 = \psi(H) && K_1 = \psi(\ker(\varphi))\\
	 S_2 = \varphi(H) && K_2 = \varphi(\ker(\psi))
\end{align*}
and \(\theta\) is given by
\begin{equation*}
	\theta\colon \frac{S_1}{K_1} \to \frac{S_2}{K_2} \colon \psi(h)K_1 \mapsto \varphi(h)K_2.
\end{equation*}

Using the structure description of pre-nests, we may prove the following converse of \zcref{prop:tcc1isnest}.

\begin{proposition}
	\label{prop:nestistcc1}
	Let \(G\) be a group and \(\nest{N}\) a nest in \(G\). Then there exist \(H \leq G \times G\) and \(\varphi,\psi \in \Hom(H,G)\) such that \(\nest{N} = [1]_{\varphi,\psi}\).
\end{proposition}
\begin{proof}
	Let \(\nest{N}\) be a nest in \(G\) and let \(\nest{N}_\times\) be a pre-nest  in \(G \times G\) inducing \(\nest{N}\). Let \(S_1,S_2,K_1,K_2,\theta\) determine \(\nest{N}_\times\) as in \zcref{thm:neststructure}. Consider the subgroup \(D \leq S_1/K_1 \times S_2/K_2\) given by
	\begin{equation*}
		D \coloneq \{ (\bar{c},\theta(\bar{c})) \mid \bar{c} \in S_1/K_1\}.
	\end{equation*}
	For every \(\bar{c} \in S_1/K_1\), fix some preimage \(c \in S_1\). Similarly, with the same abuse of notation as mentioned before, fix a preimage \(\theta(c) \in S_2\) of \(\theta(\bar{c})\). Then \(H\), the preimage of \(D\) in \(S_1 \times S_2\), is exactly the set
	\begin{equation*}
		H = \{ (ac,b^{-1}\theta(c)) \mid a \in K_1, b \in K_2, \bar{c} \in S_1/K_1 \}.
	\end{equation*}
	Let \(\psi\) and \(\varphi\) be the projections to the first and second component respectively, composed with the inclusion in \(G\), such that \(\psi,\varphi \in \Hom(H,G)\). It is now easy to see that
	\begin{equation*}
		\nest{N}_\times = \{ (\psi(h),\varphi(h)^{-1}) \mid h \in H\},
	\end{equation*}
	and therefore \(\nest{N} = [1]_{\varphi,\psi}\).
\end{proof}

Combined, \zcref{prop:tcc1isnest,prop:nestistcc1} lead us to our first main result:

\begin{theorem}[manual-num=A]
    \label{thm:mainthmA}
	A group is residually finite with respect to nests if and only if it is completely twisted conjugacy separable.
\end{theorem}

Stebe already showed that the family of NRF-groups is closed under taking subgroups and finite extensions, but he did not study quotients. We do so below.

\begin{proposition}
	\label{prop:NRFclosedquotients}
	The family of NRF-groups is closed under taking quotients.
\end{proposition}
\begin{proof}
	Let \(G\) be an NRF-group, \(N \normalsub G\), and let \(p\) be the natural projection from \(G\) to \(\bar{G} \coloneq G/N\). Suppose that \(\bar{g} \in \bar{G}\) and that \(\bar{\nest{N}}\) is a nest in \(\bar{G}\) with \(\bar{g} \notin \bar{\nest{N}}\). Let \(\bar{\nest{N}}_\times\) be a pre-nest in \(\bar{G} \times \bar{G}\) corresponding to \(\bar{\nest{N}}\). Then the set
	\begin{equation*}
		\nest{N}_\times \coloneq \{ (g_1,g_2) \in G \mid (p(g_1),p(g_2)) \in \bar{\nest{N}}_\times\}
	\end{equation*}
	is a pre-nest in \(G \times G\), and the corresponding nest \(\nest{N}\) in \(G\) is exactly \(p^{-1}(\bar{\nest{N}})\).
	Let \(g \in p^{-1}(\bar{g})\), then \(g \notin \nest{N}\) and thus there exists some \(K \finormalsub G\) such that \(g \notin \nest{N}K\). Set \(\bar{K} \coloneq p(K)\).
	
	Suppose, by contradiction, that \(\bar{g} \in \bar{\nest{N}}\bar{K}\). Then there exist \(n \in N\), \(h \in \nest{N}\) and \(k \in K\) such that \(g = nhk\). But \(nh \in \nest{N}\), so \(g \in \nest{N}K\), which contradicts the choice of \(K\). Therefore indeed \(\bar{g} \notin \bar{\nest{N}}\bar{K}\) with \(\bar{K} \finormalsub \bar{G}\).
\end{proof}

We can update \zcref{tbl:rfproperties} and Diagram \eqref{eq:diagram} with our new findings, the results can be found in \zcref{tbl:rfproperties2} and Diagram \eqref{eq:diagram2}.
\begin{equation}
	\label{eq:diagram2}
\begin{tikzcd}[row sep=small,column sep=small,ampersand replacement=\&]
	\text{NRF}\arrow[r,phantom,"\subseteq"]\arrow[dd,phantom,"=" rotate=90]\& \text{ERF}\arrow[r,phantom,"\subsetneq"] \& \text{SRF}\arrow[dr,phantom,"\subsetneq", sloped, start anchor=south east] \& \\
	\& \& \&\text{RF} \\
	\text{CTCS}\arrow[r,phantom,"\subsetneq"]\& \text{TCS}\arrow[r,phantom,"\subsetneq"]\&\text{CS} \arrow[ur,phantom,"\subsetneq" description, sloped]\& 
\end{tikzcd}
\end{equation}

\begin{table}[hb]
	\caption{Separability properties of groups (revisited)}
	\label{tbl:rfproperties2}
	\begin{tabular}{cccc}
		\toprule
		separability&  subgroup- & quotient- & closed under \\
		property & closed&closed&finite extensions\\
		\midrule
		RF &Yes	& No & Yes  \\
		CS &No & No & No\\
		TCS  &No & No & ?\\
		CTCS  &Yes & Yes & Yes\\
		SRF &No &	Yes & Yes\\
		ERF  &Yes & Yes & Yes\\
		NRF &Yes & Yes & Yes\\
		\bottomrule
	\end{tabular}
\end{table}
One may wonder whether any additional inclusions could be added to Diagram \eqref{eq:diagram2}. \zcref{ex:tcsquo} provides a group that is TCS but not SRF, so certainly we have \(\mathrm{TCS} \not\subseteq \mathrm{SRF}\). It follows that \(\mathrm{TCS} \not\subseteq \mathrm{ERF}\), \(\mathrm{CS} \not\subseteq \mathrm{SRF}\) and \(\mathrm{CS} \not\subseteq \mathrm{ERF}\). 

In \zcref{ex:tcssub} we showed that \(\SL_n(\ZZ)\) (\(n \geq 3)\) is an SRF-group that is not CS. Additionally, there exist finitely generated metabelian (and hence SRF) groups that are not CS, see e.g.\@ \cite{wehr73-a,wehr76-a}. Therefore \(\mathrm{SRF} \not\subseteq \mathrm{CS}\) and \(\mathrm{SRF} \not\subseteq \mathrm{TCS}\). The only inclusions that could still be added to the diagram, are \(\mathrm{ERF} \subseteq \mathrm{CS}\), \(\mathrm{ERF} \subseteq \mathrm{TCS}\) and \(\mathrm{ERF} \subseteq \mathrm{NRF}\).

The existence of an ERF-group that is not CS would imply that none of these inclusions hold true. Such a group, if it exists at all, seems hard to come by. In particular, none of the groups mentioned in the previous paragraph is ERF. Indeed, \(\SL_n(\ZZ)\) contains the free group \(F_2\) (which is not ERF); by a result of Jeanes and Wilson \cite{jw78-a}, a finitely generated soluble ERF-group must be polycyclic and hence CS. The following question thus appears to be presently unresolved.

\begin{question}
	Do there exist ERF-groups that are not CS, TCS or NRF?
\end{question}

\section{Nilpotent-by-finite groups}
\label{sec:nilpotent}

The study of twisted conjugacy on nilpotent-by-finite groups has proven to be very successful. As such, this family of groups is a natural starting point to study complete twisted conjugacy separability. Since finitely generated nilpotent-by-finite groups are polycyclic-by-finite, we already know they are completely twisted conjugacy separable. 

As demonstrated by Wehrfritz in \cite[Ex.~2]{wehr73-a}, there exists a residually finite nilpotent group which is not conjugacy separable. Thus, we opt to restrict ourselves to studying the family of nilpotent-by-finite groups that are strongly residually finite. These groups were characterised by Menth \cite[Thm.~5.15]{ment02-a}, see also \cite[Sec.~4]{rrv09-a} and \cite[Sec.~2]{wehr11-a} for more details and remarks.

\begin{theorem}
	\label{thm:nilperfgroup}
	Let \(G\) be a nilpotent group with torsion subgroup \(\tau(G)\). Then \(G\) is SRF if and only if the following conditions hold:
	\begin{itemize}
		\item \(G/\tau(G)\) has no quasicyclic section;
		\item each \(p\)-component of \(\tau(G)\) is abelian-by-finite and has finite exponent.
	\end{itemize}
\end{theorem}

For the reader's convenience, we provide an example of an infinitely generated nilpotent SRF-group, which appeared in both \cite[Ex.~4.1]{rrv11-a} and \cite[Lem.~2.8]{wehr11-a}.

\begin{example}
	\label{ex:nilpexample}
	The group \(G \coloneq (\bigoplus_p\, C_{p^2}) \rtimes \ZZ\), where \(\ZZ = \langle x \rangle\) acts on the cyclic group \(C_{p^2} = \langle y_p \rangle\) via \(y_p^x = y_p^{1+p}\) for each prime \(p\), is an infinitely generated \(2\)-step nilpotent SRF-group. In particular, it is not polycyclic-by-finite.
\end{example}

Smirnov proved in \cite{smir63-a} that for nilpotent groups, the properties SRF and ERF are equivalent; Menth generalised this in \cite[Thm.~7.3]{ment02-a} to nilpotent-by-finite groups. We aim to extend these results to complete twisted conjugacy separability, for which we will need some intermediate results.
 
The \zcref[nocap,noref]{lem:propP} below is inspired by \cite[Lem.~2.7]{fk22-a} and \cite[Cor.~4.3]{rrv09-a}.

\begin{lemma}
	\label{lem:propP}
	Let \(\mathcal{P}\) be a property that is retained when taking subgroups and finite extensions. Then a finite-by-\(\mathcal{P}\) RF-group also has \(\mathcal{P}\).
\end{lemma}
\begin{proof}
	Let \(G\) be residually finite and let \(N \normalsub G\) be a finite subgroup such that \(G/N\) has \(\mathcal{P}\). Take \(M \finormalsub G\) such that \(M \cap N = 1\). By the second isomorphism theorem,
	\begin{equation*}
		M \cong \frac{M}{M \cap N} \cong \frac{MN}{N} \leq \frac{G}{N},
	\end{equation*}
	hence \(M\) has \(\mathcal{P}\). Then \(G\) has \(\mathcal{P}\) as well.
\end{proof}

\begin{proposition}
    \label{prop:SRFisCTCSforquotientbycentral}
	Let \(G\) be an SRF-group. If \(C\) is a central subgroup of \(G\) such that \(G/C\) is CTCS, then \(G\) is CTCS.
\end{proposition}
\begin{proof}
	Let \(g \in G\) and \(\varphi,\psi \in \Hom(H,G)\) with \([g]_{\varphi,\psi} \neq [1]_{\varphi,\psi}\). Set \(\tilde{G} \coloneq G/C\). We consider two cases. 
	
	First, if \([\tilde{g}]_{\tvarphi,\tpsi} \neq [\tilde{1}]_{\tvarphi,\tpsi}\), then we find \(\tilde{N} \finormalsub \tilde{G}\) such that \([\bar{g}]_{\bvarphi,\bpsi} \neq [\bar{1}]_{\bvarphi,\bpsi}\) in \(\bar{G} \coloneq \tilde{G}/\tilde{N}\). Let \(N\) be the preimage of \(\tilde{N}\) in \(G\), then \(N \finormalsub G\) and hence \(\bar{G} \cong G/N\) is a finite quotient of \(G\).
	
	Second, if \([\tilde{g}]_{\tvarphi,\tpsi} = [\tilde{1}]_{\tvarphi,\tpsi}\), we have that \([g]_{\varphi,\psi} = [c]_{\varphi,\psi}\) for some \(c \in C\). Set \(D \coloneq C \cap [1]_{\varphi,\psi}\), which is a subgroup because \(C\) is central. As \(G\) is SRF and \(D\) is normal in \(G\), pick some \(M \finormalsub G\) such that \(c \notin DM\). Set \(K \coloneq M \cap C\), then \(K\) is normal in \(G\), \(K \finormalsub C\) and \(c \notin DK\). Set \(\bar{G} \coloneq G/K\) and suppose, by contradiction, that \([\bar{c}]_{\bvarphi,\bpsi} = [\bar{1}]_{\bvarphi,\bpsi}\). Then there exist \(h \in G\) and \(k \in K\) such that
	\begin{equation*}
		c = \psi(h)\varphi(h)^{-1}k.
	\end{equation*}
	But then \(\psi(h)\varphi(h)^{-1} = ck^{-1} \in C\), and hence \(\psi(h)\varphi(h)^{-1} \in D\). In turn, this means that \(c \in DK\), which cannot be the case. Therefore \([\bar{c}]_{\bvarphi,\bpsi} \neq [\bar{1}]_{\bvarphi,\bpsi}\). Now \(\bar{G}\) is finite-by-CTCS and residually finite. Using \zcref{lem:propP}, it is CTCS. We can now proceed as in the first case.
\end{proof}
In particular, if \(G\) is abelian we may take \(C = G\) in the above \zcref[nocap,noref]{prop:SRFisCTCSforquotientbycentral}, hence every abelian SRF-group is also a CTCS-group.

We now have all necessary tools to extend Menth's result to complete twisted conjugacy separability.

\begin{theorem}
	\label{thm:mainthmnilp}
	For a nilpotent-by-finite group \(G\), the following are equivalent:
	\begin{enumerate}[label=(\alph*)]
		\item \(G\) is strongly residually finite (SRF);
		\item \(G\) is extended residually finite (ERF);
		\item \(G\) is residually finite with respect to nests (NRF);
		\item \(G\) is completely twisted conjugacy separable (CTCS).
	\end{enumerate}
\end{theorem}
\begin{proof}
	The equivalence between (a) and (b) is the aforementioned result by Menth, and the equivalence between (c) and (d) is just \zcref{thm:mainthmA}.
	
    Since any subgroup is a nest, (c) implies (b), and it now suffices to prove that (b) implies (d). Let \(G\) be a nilpotent-by-finite ERF-group and let \(N \finormalsub G\) be nilpotent. Because subgroups of ERF-groups are ERF, \(N\) is ERF and hence SRF. By induction on the nilpotency class and using \zcref{prop:SRFisCTCSforquotientbycentral}, it follows that \(N\) is CTCS. Since \(G\) is a finite extension of \(N\), \(G\) itself is CTCS. So indeed (b) implies (d) and hence all of (a), (b), (c) and (d) are equivalent.
\end{proof}
	
\section{Polycyclic-by-nilpotent-by-finite groups}

At this point, we know of two families of groups that are completely twisted conjugacy separable: the polycyclic-by-finite groups and the nilpotent-by-finite SRF-groups, which are described in \zcref{thm:nilperfgroup}. Using a result by Kilsch, we can glue these results together and extend  \zcref{thm:mainthmnilp} to the family of polycyclic-by-nilpotent-by-finite groups.

A priori, the term ``polycyclic-by-nilpotent-by-finite'' could be ambiguous, as it could refer to either polycyclic-by-(nilpotent-by-finite) groups or (polycyclic-by-nilpotent)-by-finite groups. The following \zcref[nocap,noref]{prop:polybynilpbyfinite} shows that there is no such ambiguity. We use \(\gamma_i(G)\) to denote the \(i\)-th term of the lower central series of a group \(G\), i.e.\@ \(\gamma_1(G) \coloneq G\) and \(\gamma_{i+1}(G) \coloneq [\gamma_i(G),G]\).

\begin{proposition}
	\label{prop:polybynilpbyfinite}
	Let \(G\) be a group. Then \(G\) is polycyclic-by-(nilpotent-by-finite) if and only if it is (polycyclic-by-nilpotent)-by-finite.
\end{proposition}
\begin{proof}
	First assume that \(G\) is polycyclic-by-(nilpotent-by-finite). Let \(P \normalsub G\) be polycyclic normal subgroup such that \(G/P\) is nilpotent-by-finite. Then there exists a nilpotent finite index normal subgroup of \(G/P\), which must be of the form \(N/P\). This \(N\) is a finite index normal subgroup of \(G\) that is polycyclic-by-nilpotent, therefore \(G\) is (polycyclic-by-nilpotent)-by-finite.
	
	For the converse, let \(N \finormalsub G\) be polycyclic-by-nilpotent and let \(P \normalsub N\) be polycyclic such that \(N/P\) is nilpotent. If \(c\) is the nilpotency class of \(N/P\), then \(\gamma_{c+1}(N/P) = 1\). But
	\begin{equation*}
		\gamma_{c+1}(N/P) = \frac{\gamma_{c+1}(N)P}{P},
	\end{equation*}
	thus \(\gamma_{c+1}(N) \leq P\), hence \(\gamma_{c+1}(N)\) is then polycyclic. Moreover, since it is characteristic in \(N\), it is normal in \(G\). Therefore \(G/\gamma_{c+1}(N)\) contains \(N/\gamma_{c+1}(N)\) as a nilpotent finite index normal subgroup, and \(G\) is polycyclic-by-(nilpotent-by-finite).
\end{proof}

We recall a result by Kilsch on derivations. The definition used by Kilsch is that a derivation \(d\colon H \to G\) between groups \(H\) and \(G\) is a map such that \[d(h_1h_2) = d(h_1)^{h_2}d(h_2),\]
where \(H\) acts on \(G\) (from the right) via automorphisms.

\begin{theorem}[see {\cite[Thm.~D]{kils83-a}}]
	\label{thm:derivationclosed}
	The image of a derivation from a soluble-by-finite group into a polycyclic-by-finite group \(G\) is closed with respect to the finite-index topology on \(G\).
\end{theorem}

\begin{definition}
	Let \(G\) and \(H\) be groups and let \(\varphi,\psi \in \Hom(H,G)\). The \emph{coincidence subgroup} \(\Coin(\varphi,\psi)\) of \(\varphi\) and \(\psi\) is defined as
	\begin{equation*}
		\Coin(\varphi,\psi) \coloneq \{ h \in H \mid \varphi(h) = \psi(h) \}.
	\end{equation*}
\end{definition}

We now generalise \zcref{thm:mainthmnilp} to the family of polycyclic-by-nilpotent-by-finite groups, obtaining our second main result.

\begin{theorem}[manual-num=B]
	\label{thm:mainthmB}
		For a polycyclic-by-nilpotent-by-finite group \(G\), the following are equivalent:
	\begin{enumerate}[label=(\alph*)]
		\item \(G\) is strongly residually finite (SRF);
		\item \(G\) is extended residually finite (ERF);
		\item \(G\) is residually finite with respect to nests (NRF);
		\item \(G\) is completely twisted conjugacy separable (CTCS).
	\end{enumerate}
\end{theorem}
\begin{proof}
	It suffices to prove that \(G\) being SRF implies it is CTCS. Let \(g \in G\), let \(H\) be a group and let \(\varphi,\psi \in \Hom(H,G)\) be such that \([g]_{\varphi,\psi} \neq [1]_{\varphi,\psi}\). We may assume that \(H \leq G \times G\) and therefore that \(H\) is soluble-by-finite. Let \(N \normalsub G\) be a polycyclic normal subgroup such that \(\bar{G} \coloneq G/N\) is nilpotent-by-finite, and let \(\bvarphi,\bpsi \in \Hom(H,\bar{G})\). If \([\bar{g}]_{\bvarphi,\bpsi} \neq [\bar{1}]_{\bvarphi,\bpsi}\), then we need only apply \zcref{thm:mainthmnilp}.
	
	If however \([\bar{g}]_{\bvarphi,\bpsi} = [\bar{1}]_{\bvarphi,\bpsi}\), there exists some \(n \in N\) such that
	\([g]_{\varphi,\psi} = [n]_{\varphi,\psi}\). Now consider the map 
    \begin{equation*}
        d\colon \Coin(\bvarphi,\bpsi) \to N\colon h \mapsto \psi(h)^{-1}\varphi(h).
    \end{equation*}
    For \(h_1,h_2 \in \Coin(\bvarphi,\bpsi)\), one has that
    \begin{align*}
        d(h_1h_2) &= \psi(h_1h_2)^{-1}\varphi(h_1h_2)\\
                &= \psi(h_2)^{-1}\psi(h_1)^{-1}\varphi(h_1)\varphi(h_2)\\
                &= \psi(h_2)^{-1}d(h_1)\psi(h_2)\psi(h_2)^{-1}\varphi(h_2)\\
                &= d(h_1)^{h_2}d(h_2),
    \end{align*}
    where \(\Coin(\bvarphi,\bpsi)\) acts on \(N\) (from the right) via
    \begin{equation*}
    	\Coin(\bvarphi,\bpsi) \to \Aut(N)\colon h \mapsto \iota_{\psi(h)^{-1}}.
    \end{equation*}
    Thus, it follows that \(d\) is a derivation, and its image is exactly \(N \cap [1]_{\varphi,\psi}\). Since \(N\) is polycyclic and \(\Coin(\bvarphi,\bpsi) \leq H\) is soluble-by-finite, we can apply \zcref{thm:derivationclosed}: choose \(K \finormalsub N\) such that \(n \notin (N \cap [1]_{\varphi,\psi})K\). Since \(N\) is finitely generated, we may assume \(K\) is characteristic in \(N\) and therefore normal in \(G\). Set \(\tilde{N} \coloneq N/K\) and \(\tilde{G} \coloneq G/K\), then \([\tilde{g}]_{\tvarphi,\tpsi} = [\tilde{n}]_{\tvarphi,\tpsi} \neq [\tilde{1}]_{\tvarphi,\tpsi}\).
    
    As \(\tilde{N}\) is finite and \(\tilde{G}/\tilde{N} \cong G/N = \bar{G}\) is nilpotent-by-finite, \(\tilde{G}\) is a finite-by-(nilpotent-by-finite) SRF-group. From \zcref{lem:propP} we obtain that \(\tilde{G}\) is then nilpotent-by-finite, hence we can finish by again applying \zcref{thm:mainthmnilp}.
\end{proof}

In \cite[Thm.~4]{malc58-a}, Mal'cev proved that the direct product of two ERF-groups is itself ERF, and in \cite[Thm.~4]{ag73-a} Allenby and Gregorac proved that the semidirect product \(N \rtimes H\) of a finitely generated ERF-group \(N\) and an ERF-group \(H\) is itself ERF. Combined with \zcref{thm:mainthmB}, these results can be used to construct polycyclic-by-nilpotent-by-finite CTCS-groups that are neither polycyclic-by-finite nor nilpotent-by-finite.

\begin{corollary}
	\label{cor:pcbynilpsemiCTCS}
	Let \(G \coloneq N \rtimes H\) be the (semi)direct product of a polycyclic group \(N\) and a nilpotent-by-finite ERF-group \(H\). Then \(G\) is completely twisted conjugacy separable.
\end{corollary}

We now use this \zcref[nocap,noref]{cor:pcbynilpsemiCTCS} to construct some explicit examples.

\begin{example}
	Let \(N\) be the group \(\ZZ^2 \rtimes_\theta \ZZ\) with \(\theta\colon \ZZ \to \Aut(\ZZ^2) \cong \GL_2(\ZZ)\) determined by
	\begin{equation*}
		\theta(1) = \begin{bmatrix}
			2 & 1\\
			1 & 1
		\end{bmatrix},
	\end{equation*}
	which is polycyclic but not nilpotent-by-finite (see e.g.\@ \cite[Ex.~3.4]{gw03-a}). Let \(H\) be the nilpotent group from \zcref{ex:nilpexample}, which is not polycyclic-by-finite since it is not finitely generated. Then the direct product \(N \times H\) is a polycyclic-by-nilpotent CTCS-group.
\end{example}

\begin{example}
	Let \(H\) again be the nilpotent group from \zcref{ex:nilpexample}. Consider the homomorphism \(\lambda\colon H \to \Aut(\ZZ^2) \cong \GL_2(\ZZ)\), defined by
	\begin{equation*}
		\lambda(x) = \begin{bmatrix}
			2 & 1\\
			1 & 1
		\end{bmatrix},
		\quad
		\lambda(y_2) = \begin{bmatrix*}[r]
		-1 & 0\\
		0 & -1
		\end{bmatrix*},
		\quad
		\lambda(y_p) = \begin{bmatrix}
			1 & 0\\
			0 & 1
		\end{bmatrix} \textrm{ when } p \geq 3.
	\end{equation*}
	Then \(\ZZ^2 \rtimes_\lambda H\) is a polycyclic-by-nilpotent CTCS-group. But this group contains both \(N\) (as above) and \(H\) as subgroups, so it is neither nilpotent-by-finite nor polycyclic-by-finite.
\end{example}

\section*{Acknowledgements}
The author is indebted to Peter Wong for bringing the work by Stebe on nests to his attention, to Karel Dekimpe for catching a mistake in the proof of \zcref{thm:mainthmB} and helping to correct said proof, and to the anonymous referee for a multitude of helpful comments and valuable suggestions.

\printbibliography

\end{document}